\documentclass[10pt, reqno]{amsart}
\pdfoutput=1 
\usepackage{amsmath, amssymb, amsthm, enumerate, bbm}
\usepackage[colorlinks, linkcolor=blue, citecolor=magenta, hypertexnames=false,backref]{hyperref} 
\usepackage[dvipsnames]{xcolor}

\usepackage{graphicx}
\usepackage{booktabs}
\usepackage[font=small]{caption} 
\usepackage[flushleft]{threeparttable}
\usepackage{mathtools, breqn}
\usepackage{url}

\def\spine{1.1in}
\usepackage[
heightrounded,
top=1in,
bottom=1.3in,
inner=\spine,
outer=\spine 
]{geometry} %

\usepackage[final, expansion=true, protrusion=true, spacing=true, kerning=true]{microtype}
\microtypecontext{spacing= nonfrench}

\def\supp{\operatorname{supp}}
\def\theta{\vartheta}
\def\Tr{\operatorname{Tr}}

\numberwithin{equation}{section}

\newtheorem{theorem}{Theorem}[section]
\newtheorem{lemma}[theorem]{Lemma}

\newtheorem{proposition}[theorem]{Proposition}
\newtheorem{corollary}[theorem]{Corollary}

\theoremstyle{definition}

\newtheorem{remark}{Remark}

\title{Optimizers of three-point energies and nearly orthogonal sets}
\date{\today}
\author[D. Bilyk]{Dmitriy Bilyk}
\author[D. Ferizovi\'{c}]{Damir Ferizovi\'{c}}
\author[A. Glazyrin]{Alexey Glazyrin}
\author[R. Matzke]{Ryan W. Matzke}
\author[J. Park]{Josiah Park}
\author[O. Vlasiuk]{Oleksandr Vlasiuk}
\address{School of Mathematics, University of Minnesota, Minneapolis, MN 55455} 
\email{dbilyk@math.umn.edu}
\address{Department of Mathematics, Katholieke Universiteit Leuven, Leuven, Belgium} 
\email{damir.ferizovic@kuleuven.be}
\address{School of Mathematical \& Statistical Sciences, The University of Texas Rio Grande Valley, Brownsville, TX 78520}
\email{alexey.glazyrin@utrgv.edu}
\address{Department of Mathematics, Vanderbilt University, Nashville, TN 37240} 
\email{ryan.w.matzke@vanderbilt.edu}
\address{School of Mathematics, Georgia Insititute of Technology, Atlanta, GA 30332}
\email{j.park@gatech.edu}
\address{Department of Mathematics, Vanderbilt University, Nashville, TN 37240}
\email{oleksandr.vlasiuk@vanderbilt.edu }


\begin{document} 
\maketitle

\maketitle
\begin{abstract}
This paper is devoted to spherical measures and point configurations optimizing three-point energies. Our main goal is to extend the classic optimization problems based on pairs of distances between points to the context of three-point potentials. In particular,  we study three-point analogues of the sphere packing problem and the optimization problem for $p$-frame energies based on three points. It turns out  that both problems are inherently connected to the problem of nearly orthogonal sets by Erd\H{o}s. As the outcome, we provide a new solution of the Erd\H{o}s problem from the three-point packing perspective. We also show that the orthogonal basis uniquely minimizes the $p$-frame three-point energy when $0<p<1$ in all dimensions. The arguments make use of multivariate polynomials employed in semidefinite programming and based on the classical Gegenbauer polynomials. For $p=1$, we completely solve the analogous problem on the circle. As for higher dimensions, we show that the Hausdorff dimension of minimizers is not greater than $d-2$ for measures on $\mathbb{S}^{d-1}$. As the main ingredient of our proof, we show that the only isotropic measure without obtuse angles is the uniform distribution over an orthonormal basis.
\end{abstract}

\tableofcontents

\section{Introduction}

Various interesting point configurations in metric spaces are realized as optimizers of energies defined by two-point potentials. For an $N$-point configuration $\omega_N$ on the unit sphere, $\mathbb{S}^{d-1}$, the energy is defined as

\begin{equation}\label{e.2ener}
E_K (\omega_N) = \frac{1}{N^2} \sum_{x,y \in \omega_N}  K (x, y),
\end{equation}
where $K$ is the function defining a two-point potential. Typically, $K$ depends on the Euclidean distance between points $x$ and $y$.  A very interesting example, closely connected to unit norm tight frames and spherical designs, is  the $2$-frame potential $K(x,y)=|\langle x,y \rangle|^2$, or simply \textit{frame potential},  introduced by Benedetto and Fickus \cite{BF}, and later generalized, by Ehler and Okoudjou \cite{EO}, to the \textit{$p$-frame potential} $|\langle x,y \rangle|^p$, for $p \in (0, \infty)$, which has been studied further in, e.g., \cite{BE, BCGKO, BGMPV1, BGMPV2, CGGKO, Gl, GP, H1, H2, WO, XX}. See \cite{BHS} for an extensive introduction on energy optimization for two-point potentials.

In this paper we study three-point potentials, i.e. analogous   energies depending on interactions of triples of points rather than  pairwise interactions. That is, we shall be concerned with the minimization of three-point discrete energies and energy integrals:
$$ E_K (\omega_N) = \frac{1}{N^3} \sum_{x,y,z \in \omega_N}  K (x, y,z) \,\,\, \textup{ and } \,\,\, I_K(\mu) = \int_{\mathbb{S}^{d-1}} \int_{\mathbb{S}^{d-1}} \int_{\mathbb{S}^{d-1}} K(x,y,z) d\mu(x) d\mu(y) d\mu(z) ,$$ 
see Section \ref{sec:not} for precise definitions.  Such energies arise  in various applications and have been  previously studied by the authors \cite{BFGMPV,BFGMPV1}.

The notion of the $p$-frame potential is naturally extended  to the three-point case by defining  the three-point $p$-frame potential as $K(x,y,z) = |uvt|^p = |\langle y,z \rangle \langle x,z \rangle \langle x,y \rangle|^p$.  In this paper, we are mostly interested in the case $p=1$, as well as $0<p<1$. As we shall see, even in this case the problem of describing minimizers appears to be highly nontrivial and inherently connected to the problem of Erd\H{o}s about nearly orthogonal sets and  packing problems which we describe below. 

\subsection{Nearly orthogonal sets} A (multi-)set of nonzero vectors in $\mathbb{R}^d$ is called \textit{nearly orthogonal} if, for any three distinct vectors from the set, at least two of them are orthogonal. Erd\H{o}s asked about the maximum size of nearly orthogonal sets. In 1991, Rosenfeld showed that the maximum size of such a set is $2d$ \cite{Ros}. Other proofs of this result are given in \cite{Pud} and \cite{Dea}, and we provide an alternative proof in Theorem \ref{thm:rosen}. Nearly orthogonal sets of size $2d$ in $\mathbb{S}^{d-1}$ are called \textit{Rosenfeld sets}. One might be tempted to believe that Rosenfeld sets  in $\mathbb{R}^d$ are necessarily unions of two orthogonal bases but, in fact, for $d\geq 5$ this is not true (see \cite{Dea} for examples) and the full characterization of Rosenfeld sets is an open problem.

We would also like to remark that Rosenfeld \cite{Ros} used the term ``almost orthogonal sets'', but ``nearly orthogonal'' seems to be more common in this context in later literature, although both are still used (see, e.g., \cite{AlSz,Pol}).

\subsection{Packing problems on the sphere} The classic packing problem on the sphere, or the problem of determining an optimal spherical code, consists in finding the maximal number of points in the configuration $\omega_N\in\mathbb{S}^{d-1}$ such that for any distinct points $x,y\in\omega_N$, $\langle x,y \rangle \leq \alpha$. This condition ensures that distinct points of $\omega_N$ are separated by a spherical distance of at least $\arccos \alpha$. In the case $\alpha=1/2$, the problem is known as a \textit{kissing number} problem that, despite its rich history, is solved only in a handful of small dimensions \cite{SvdW, Lev, OS, Mus}. 

We study the three-point analogue of the packing problem. Particularly, we are interested in finding the maximal number of points in the configuration $\omega_N\in\mathbb{S}^{d-1}$ such that for any distinct points $x,y,z\in\omega_N$, $\langle x,y \rangle \langle x,z \rangle \langle y,z \rangle \leq \alpha$. In Section \ref{sec:erdos}, we will show that the three-point packing problem for $\alpha=0$ is surprisingly similar to the classic two-point packing problem for $\alpha=0$ and the optimal configurations are precisely Rosenfeld sets.

\subsection{Outline of the paper} In Section \ref{sec:not} we present some relevant background information, in particular, about isotropic measures, frames, frame energy, and semidefinite programming bounds.  In Section \ref{sec:erdos} we explore the connections between nearly orthogonal sets and three-point packing bounds, giving an alternative proof and generalizing Rosenfeld's result (Theorem \ref{thm:rosen}). Section \ref{sec:pframe} is devoted to the three-point $p$-frame energy: 
in particular, we show that in all dimensions for $0<p<1$ any minimizer of the $p$-frame energy  is a uniform distribution over an orthonormal basis, up to symmetries (Theorem \ref{thm:p<1}) and on $\mathbb S^1$ minimizers of the $1$-frame energy are convex combinations of uniform distributions over two orthonormal bases (Theorem \ref{thm:2-dim}). In Section \ref{sec:non-obtuse} we use a spherical version of Jung's inequality \cite{D} to show that any isotropic measure without obtuse angles in its support  is a uniform distribution over an orthonormal basis. Finally, in Section \ref{sec:supp} we show that the support of a measure minimizing the three-point $1$-frame energy must have codimension at least one (Theorem \ref{thm:codim1}).

\section{Notation and preliminaries}\label{sec:not}

The notation in the paper follows \cite{BFGMPV, BFGMPV1}. All potentials are defined for triples of points $(x,y,z)$ 
on the unit sphere $\mathbb{S}^{d-1}$. For brevity, throughout the paper we use the notation $u=\langle y,z \rangle$, $v=\langle x,z \rangle$, $t=\langle x,y \rangle$, where $\langle \cdot , \cdot \rangle$ denotes the standard Euclidean inner product. We denote by $\mathcal{P}(\mathbb{S}^{d-1})$ the set of Borel probability measures on $\mathbb{S}^{d-1}$. The normalized uniform distribution over $\mathbb{S}^{d-1}$ is denoted by $\sigma$. A measure $\mu$ is called \textit{balanced} if its center of mass is at the origin, that is, $\int_{\mathbb{S}^{d-1}} x d\mu(x) =0$.

\subsection{Three-point energies} Let $\omega_N = \{ z_1, z_2, ..., z_N\}$ be an $N$-point configuration (multiset) in $\mathbb{S}^{d-1}$, for $N \geq 3$. Given a continuous three-point kernel $K:\Big(\mathbb{S}^{d-1}\Big)^3 \rightarrow \mathbb{R}$, the discrete $K$-energy of $\omega_N$ is defined to be 
\begin{equation}\label{eq:DiscreteEnergyDef}
E_K(\omega_N) :=  \frac{1}{N^3} \sum_{i=1}^{N} \sum_{j=1}^{N} \sum_{k = 1}^{N} K(z_i, z_j, z_k).
\end{equation}
Similarly, we define the energy for a measure $\mu\in\mathcal{P}(\mathbb{S}^{d-1})$:
\begin{equation}\label{eq:ContEnergyDef}
I_K(\mu) = \int_{\mathbb{S}^{d-1}} \int_{\mathbb{S}^{d-1}} \int_{\mathbb{S}^{d-1}} K(x,y,z) d\mu(x) d\mu(y) d\mu(z).
\end{equation}
Under the normalization in \eqref{eq:DiscreteEnergyDef}, $ E_K (\omega_N) = I_K (\mu)$ for the discrete measure $\mu = \frac{1}{N}\sum_{x\in \omega_N} \delta_x$. 
Generally, we are interested in finding point configurations and measures optimizing the energy for a given potential $K$.

\subsection{Isotropic energies and frame potentials} A measure $\mu\in\mathcal{P}(\mathbb{S}^{d-1})$ is called \textit{isotropic} if
\begin{equation}\label{eqn:isotropic_init}
\int_{ \mathbb{S}^{d-1}} x x^T \,  d\mu(x)=\frac 1 d I_d,
\end{equation}
where $I_d$ is the unit $d\times d$ matrix.
This is equivalent to the statement that for any $y\in\mathbb{S}^{d-1}$,
\begin{equation}\label{eqn:isotropic}
\int_{ \mathbb{S}^{d-1}} \langle x, y\rangle^2\,  d\mu(x)=\frac 1 d.
\end{equation}

\noindent One can easily show that isotropic measures are precisely the minimizers of the 2-frame energy, see e.g. \cite{BGMPV1}. 
\begin{lemma}\label{lem:frame}
For any $\mu\in \mathcal{P}(\mathbb{S}^{d-1})$,
\begin{equation}\label{eq:frame}
\int_{ \mathbb{S}^{d-1}} \int_{\mathbb{S}^{d-1}} \langle x,y \rangle^2 d\mu(x) d\mu(y) \geq \frac 1 d.
\end{equation}
The equality holds if and only if $\mu$ is isotropic.
\end{lemma}

 Applying this lemma to the uniform distribution over a discrete configuration $\omega_N=\{x_1,\ldots,x_N\}$, $N \geq d$, we get the result of Benedetto and Fickus in the discrete setting \cite[Theorem 7.1]{BF}:
\begin{equation}\label{eq:framediscrete}
\sum_{i=1}^{N}\sum_{j = 1}^{N} \langle x_i, x_j \rangle^2\geq \frac {N^2} d,
\end{equation}
with equality if and only if for each $y\in \mathbb R^d$ $$ \sum_{i=1}^N   \langle x_i, y  \rangle^2 = \frac{N}{d}  \|y \|^2 ,$$
i.e.  $\omega_N$ is  a \textit{unit norm tight frame}. 
  The bound \eqref{eq:framediscrete} is essentially a special case of the results of Welch \cite{W} and Sidelnikov \cite[Corollary 1]{Sid}.

\subsection{Semidefinite programming and three-point bounds}\label{sec:SemiDef Subsec}

For a variety of optimization problems for two-point energies, the linear programming method serves as the main machinery (see, e.g.,  \cite{Del, KL, Y, CK, BGMPV1}). 
Naturally, optimization problems for three-point energies should employ the three-point generalization of this method. This generalization was developed for the spherical case by Bachoc and Vallentin \cite{BV} who used it to obtain new bounds for the kissing number problem. They produced a class of infinite matrices and associated polynomials of the form
\begin{equation}\label{eq:SemDefProgYs}
(Y_{m}^d)_{i+1,j+1} (x,y,z) := Y_{m,i,j}^d(x,y,z) := P_i^{d + 2m}(u) P_{j}^{d +  2m}(v) Q_m^d(u,v,t),
\end{equation}
where $m,i,j \in \mathbb{N}_0$, $P_m^h$ is the normalized Gegenbauer polynomial of degree $m$ on $\mathbb{S}^{h-1}$ and
\begin{equation}\label{eq:BachocValQKernel}
Q_m^d(u,v,t) = ((1-u^2)(1-v^2))^{\frac{m}{2}}P_{m}^{d-1}\left(\frac{t-uv}{\sqrt{(1-u^2)(1-v^2)}}\right).
\end{equation}

For convenience, we include here the upper left $3\times 3$, $2 \times 2$, and $1 \times 1$ submatrices of infinite matrices $Y_0^d$, $Y_1^d$, and $Y_2^d$:
$$\begin{pmatrix} 
1 & v & \frac {dv^2-1} {d-1}\\
u & uv & u \frac {dv^2-1} {d-1}\\ 
\frac {dv^2-1} {d-1} & \frac {du^2-1} {d-1} v & \frac {du^2-1} {d-1} \frac {dv^2-1} {d-1}\end{pmatrix}$$
$$ \begin{pmatrix} t-uv & u (t-uv)\\ 
v(t-uv) & uv(t-uv)\end{pmatrix}, \begin{pmatrix} \frac {(d-1)(t-uv)^2-(1-u^2)(1-v^2)} {d-2}\end{pmatrix}.$$

Symmetrizing by all permutations $\pi$ over the variables $x$, $y$, and $z$, Bachoc and Vallentin defined the following symmetric matrices and associated polynomials
\begin{equation*}\label{eq:Symmetric Semi Def Kernel}
(S_{m}^d)_{i+1,j+1} (x,y,z) := S_{m,i,j}^d (x,y,z) := \frac{1}{6} \sum_{\pi} Y_{m,i,j}^d (\pi(x), \pi(y), \pi(z)).
\end{equation*}

The following property of these matrices allows one to use them for optimization purposes. For any $\mu \in \mathcal{P}(\mathbb{S}^{d-1})$ and $e \in \mathbb{S}^{d-1}$, the infinite matrices 
$$ \int_{\mathbb{S}^{d-1}} \int_{\mathbb{S}^{d-1}} Y_{m}^d(x,y,e) d\mu(x) d \mu(y)$$
and
$$ S_m^d(\mu) := \int_{\mathbb{S}^{d-1}} \int_{\mathbb{S}^{d-1}} \int_{\mathbb{S}^{d-1}} S_{m}^d(x,y,z) d\mu(x) d \mu(y) d \mu(z)$$
are positive semidefinite, that is, all principal minors (formed by finite submatrices) are non-negative. 
This property leads to the following energy minimization  theorem {(more details and justification can be found in Sections 3 and 4 of \cite{BFGMPV1}).}

\begin{theorem}\label{thm:SemiDefMin}
Let $n \in \mathbb{N}_0$. For each $m \leq n$, let $A_m$ be an infinite, symmetric, positive semidefinite matrix with finitely many nonzero entries, with the additional requirement that $A_0$ has only zeros in its first row and first column. Let
$$K(x,y,z) = \sum_{m=0}^{n} \Tr( S_{m}^d(x,y,z)\, A_m).$$
Then $\sigma$ is a minimizer of $I_K$ over probability measures on the sphere $\mathbb{S}^{d-1}$ and $I_K(\sigma)=0$.
\end{theorem}

Observe that when the matrices $A_m$ are diagonal, the kernel $K$ is simply a non-negative linear combination of the diagonal entries of the matrices $S_m^d$. 
Due to the nature of the bounds from Theorem \ref{thm:SemiDefMin} and their most common use, we will refer to them as semidefinite programming bounds throughout the paper.

\section{Three-point packing problem and nearly orthogonal sets}\label{sec:erdos}

We use semidefinite programming bounds to prove a three-point packing bound generalizing the aforementioned result of Rosenfeld \cite{Ros}. Observe that this also gives an alternative proof of Rosenfeld's result. 

\begin{theorem}\label{thm:rosen}
Assume for any three distinct points $x, y, z$ of the set $\omega_N \subset\mathbb{S}^{d-1}$,
$$\langle x,y \rangle \langle x,z \rangle \langle y,z \rangle \leq 0.$$
Then $N \leq 2d$. The equality $N= 2d$ is achieved  only if $\omega_N$ is nearly orthogonal, i.e. a Rosenfeld set. 
\end{theorem}

\begin{proof}
Following the notation from Section \ref{sec:SemiDef Subsec}, we use 
$$6 \frac{(d-1)^2}{d^2} S_{0,2,2}^d =2(u^2v^2+u^2t^2+v^2t^2)-\frac 4 d(u^2+v^2+t^2)+\frac 6 {d^2}$$
and
$$6 \: S_{1,1,1}^d =6uvt-2(u^2v^2+u^2t^2+v^2t^2).$$

Summing up the values of the kernels above over all points of $\omega_N = \{ x_1,\dots,x_N\}$, we must get a non-negative number due to Theorem \ref{thm:SemiDefMin}, so
\begin{align*}
0 &\leq \sum\limits_{i,j,k=1}^N 6 \left(\frac{(d-1)^2}{d^2} S_{0,2,2}^d + S_{1,1,1}^d\right) (x_i, x_j, x_k)\\
 &= \sum\limits_{i,j,k=1}^N \left(6\langle x_i,x_j \rangle \langle x_i,x_k \rangle \langle x_j,x_k\rangle - \frac 4 d (\langle x_i,x_j \rangle^2 + \langle x_i,x_k \rangle^2 + \langle x_j,x_k \rangle^2) + \frac 6 {d^2}\right).
\end{align*}

All triple product terms for distinct $i,j,k$ are nonpositive so we can eliminate them without changing the validity of the inequality (equality in this step is achieved if and only if $\omega_N$ is nearly orthogonal). The remaining terms are split into four cases: $i=j=k$, $i=j\neq k$, $i=k\neq j$, and $j=k\neq i$. The sum of triple product terms is $N$ for the first group and $\sum\limits_{i,j=1}^N \langle x_i,x_j \rangle^2 - N$ for the other three groups. The same sum shows up for double product terms. Overall, we have 
\begin{align*}
0 & \leq 6\left(N+3\left(\sum\limits_{i,j=1}^N \langle x_i,x_j \rangle^2 - N\right)\right) -\frac 4 d 3N \sum\limits_{i,j=1}^N \langle x_i,x_j \rangle^2 + \frac 6 {d^2} N^3 \\
&= \left(18-\frac {12N} d\right) \sum\limits_{i,j=1}^N \langle x_i,x_j \rangle^2 - 12N + \frac {6N^3} {d^2}.
\end{align*}

If $18-\frac {12N} d$ is non-negative, then $N\leq \frac 3 2 d$ so the statement of the theorem is true. If it is negative, we note that $\sum\limits_{i,j=1}^N \langle x_i,x_j \rangle^2$ is the frame energy of the set and must be at least $\frac {N^2} d$ by the discrete version of Lemma \ref{lem:frame}. Substituting this value in the inequality, we obtain
$$\left(18-\frac {12N} d\right) \frac {N^2} d - 12N + \frac {6N^3} {d^2} \geq 0.$$
Dividing by $6N$ and factoring, we reach the inequality
$$(2-N/d)(N/d-1)\geq 0$$
that immediately implies the statement of the theorem.

Finally, we note that the size of $\omega_N$  is exactly $2d$ only when it is nearly orthogonal, i.e. a Rosenfeld set (and the converse is obvious).  
\end{proof}

Further analyzing the case of equality in the proof of the  theorem above, we observe that the frame energy of  the set $\omega_N$  is necessarily $N^2/d$, which according to \eqref{eq:framediscrete} means that $\omega_N$ has to be a tight frame. 

\begin{corollary}\label{cor:rosenframe}
Each nearly orthogonal set on $\mathbb S^{d-1}$ of size $2d$ (Rosenfeld set) is   a tight frame. 
\end{corollary}


Theorem \ref{thm:rosen} mirrors the classic packing bound: the size of a configuration in $\mathbb{S}^{d-1}$ with $\langle x,y \rangle \leq 0$ for any two distinct points is no greater than $2d$ and the set of vertices of the $d$-dimensional crosspolytope is the unique (up to orthogonal transformations) set satisfying this bound \cite{A,Sz,Ran}. This packing bound complements the initial result of Davenport and Haj\'{o}s who showed that the size of a configuration in $\mathbb{S}^{d-1}$ with $\langle x,y \rangle< 0$ for any two distinct points is no greater than $d+1$ \cite{DH}. The result of Davenport and Haj\'{o}s can be easily extended to the three-point case.

\begin{theorem}\label{thm:hd}
Assume for any three distinct points $x, y, z$ of the set $\omega_N \subset\mathbb{S}^{d-1}$,
$$\langle x,y \rangle \langle x,z \rangle \langle y,z \rangle < 0.$$
Then $N \leq d+1$.
\end{theorem}

\begin{proof}
Fix a point $x \in \omega_N$. Since changing any point $y \in \omega_N$ to its opposite $-y$ does not affect the condition of the theorem, we may, without loss of generality, assume that $\langle x, y\rangle <0$ for each $y \in \omega_N \setminus \{ x \}$.  The condition $\langle x,y \rangle \langle x,z \rangle \langle y,z \rangle < 0$ then implies that $\langle y, z\rangle <0$ for any distinct  $y, z  \in \omega_N \setminus \{ x \}$. Thus, all the inner products  in $\omega_N$ are negative. By the result of Davenport and Haj\'{o}s \cite{DH} mentioned above, the size of the set is no greater than $d+1$.
\end{proof}

Finally, it is a simple fact that if  $\langle x, y\rangle < - \varepsilon$ for each distinct $x,y \in \omega_N$, then $N \le 1 + \frac{1}{\varepsilon}$, which is independent of the dimension $d$. This can be quickly derived from the inequality $\| \sum x_i \|^2 \ge 0$. This fact also easily generalizes to the multivariate setting. 

\begin{lemma}
Let $\varepsilon >0$. Assume for any three distinct points $x, y, z$ of the set $\omega_N \subset\mathbb{S}^{d-1}$,
$$\langle x,y \rangle \langle x,z \rangle \langle y,z \rangle \le  - \varepsilon.$$
Then $N \leq 1 + \frac{1}{\varepsilon}$.
\end{lemma}

\begin{proof}
As the value of $\langle x,y \rangle \langle x,z \rangle \langle y,z \rangle$ doesn't change when any of the vectors is changed to its opposite, arguing as in the proof of Theorem \ref{thm:hd}, we may assume that all inner products between distinct elements of $\omega_N$ are negative. But $ |\langle x,y \rangle \langle x,z \rangle \langle y,z \rangle | \le | \langle x,y \rangle |$, therefore $\langle x, y\rangle < - \varepsilon$ for each distinct $x,y \in \omega_N$. Thus, the aforementioned fact implies the statement of the theorem. 
\end{proof}

The  Erd\H{o}s problem on nearly orthogonal sets gave rise to active  investigations of orthogonal representations of graphs, see e.g. \cite{Dea}.
Since the proof of Theorem  \ref{thm:rosen} demonstrates that the machinery of semidefinite bounds is effective for the Erd\H{o}s problem, it would be interesting  to find out whether any other problems regarding orthogonal representations of graphs can be solved in a similar manner.  A clear obstacle to this approach is the complexity of functions  \eqref{eq:BachocValQKernel} involved in $k$-point semidefinite bounds.

\section{\texorpdfstring{Minimal energy for multivariate $p$-frame potentials}{Minimal energy for multivariate p-frame potentials}}\label{sec:pframe}

We now turn our attention to the multivariate $p$-frame energy, i.e. the energy with the potential  $K(x,y,z) = | \langle x,y \rangle \langle x,z \rangle \langle y,z \rangle|^p = |uvt|^p $ for $p>0$.  {While this section primarily focuses on  the case $p \leq 1$, we first  quickly address the case when the $p$-frame potential is a polynomial. Corollary 5.2 in \cite{BFGMPV}  implies 
\begin{proposition}
If $p \in 2 \mathbb{N}$, then $\sigma$ minimizes the three-point $p$-frame energy.
\end{proposition}}

As the first step in understanding the set of minimizers for $p=1$, we give a description of minimizing measures for  the potential  $uvt$.

\begin{lemma}\label{lem:uvt}
For any $\mu\in\mathcal{P}(\mathbb{S}^{d-1})$,
$$I_{uvt} (\mu ) = \int_{\mathbb{S}^{d-1}}\int_{\mathbb{S}^{d-1}}\int_{\mathbb{S}^{d-1}} \langle x,y \rangle  \langle y,z \rangle\, \langle  z, x \rangle d\mu(x) d\mu(y) d\mu(z) \geq \frac 1 {d^2}$$
and the equality holds if and only if $\mu$ is isotropic.
\end{lemma}

\begin{proof} We shall give two different proofs of this fact, as both are quite instructive: one based on linear algebra and another one based on the semidefinite programming bounds of Theorem \ref{thm:SemiDefMin}. 

\noindent {\it{Proof 1 (linear algebra).}} If we denote $ x = \big( x^{(1)}, \dots, x^{(d)} \big)$, then 
$\displaystyle{uvt = \sum_{k,l,m = 1}^d x^{(k)} x^{(l)} y^{(l)} y^{(m)} z^{(m)} z^{(k)}}$.  Therefore, setting $a_{k,l} = \int_{\mathbb S^{d-1}}  x^{(k)} x^{(l)} d\mu (x)$, we see that $$ I_{uvt} (\mu ) =  \sum_{k,l,m = 1}^d a_{k,l} a_{l,m} a_{m,k} = \operatorname{Tr} (A^3),$$ where $A = (a_{k,l}) = \int_{\mathbb S^{d-1}} x x^T d\mu(x)$ is obviously positive semidefinite. 

For any positive semidefinite $d\times d $ matrix $A$ with eigenvalues $\lambda_i \ge 0$, using H\"{o}lder's inequality, one obtains $$ \operatorname{Tr} (A) = \sum_{i=1}^d \lambda_i \le  \left( \sum_{i=1}^d \lambda_i^3  \right)^{\frac{1}{3}} \cdot d^{\frac{2}{3}} = \Big( \operatorname{Tr} (A^3) \Big)^{\frac{1}{3}} \cdot d^{\frac{2}{3}} , \,\,\, \textup{ i.e. } \,\,  \operatorname{Tr} (A^3) \ge \frac{( \operatorname{Tr} (A))^3}{d^2},  $$
with equality if and only if $A$ is a multiple of the identity. 

In our case, $\displaystyle{\operatorname{Tr} (A) = \sum_{i=1}^d \int_{\mathbb S^{d-1}} \Big( x^{(i)} \Big)^2 d\mu (x) = 1}$, therefore $I_{uvt} (\mu ) =   \operatorname{Tr} (A^3) \ge \frac{1}{d^2}$, and the equality is achieved if and only if $ A = \frac{1}{d} I_d$, i.e. $\mu$ is isotropic. \\

\noindent {\it{Proof 2 (semidefinite programming).}} From the proof of Theorem \ref{thm:rosen}, we see that
$$uvt=\frac{(d-1)^2}{d^2} S_{0,2,2}^d +  S_{1,1,1}^d+\frac 2 {3d} (u^2+v^2+t^2)-\frac 1 {d^2}.$$
When integrating the right hand side, the values for $S_{0,2,2}^d$ and $S_{1,1,1}^d$ are non-negative due to Theorem \ref{thm:SemiDefMin}. The integral of $u^2+v^2+t^2$ is at least $\frac 3 d$ by Lemma \ref{lem:frame}. Overall, we get
$$\int_{\mathbb{S}^{d-1}}\int_{\mathbb{S}^{d-1}}\int_{\mathbb{S}^{d-1}} \langle x,y \rangle \langle x,z \rangle \langle y,z \rangle\, d\mu(x) d\mu(y) d\mu(z) \geq \frac 2 {3d} \cdot \frac 3 d - \frac 1 {d^2} = \frac 1 {d^2}.$$
The equality can hold only when $\mu$ is isotropic due to Lemma \ref{lem:frame}. On the other hand,  integrals of both $S_{0,2,2}^d$ and $S_{1,1,1}^d$ vanish on isotropic measures, so the energy is precisely $\frac 1 {d^2}$ for all isotropic measures.
\end{proof}
We remark that this result (although without the full characterization of minimizers) has also been proved in Corollary 5.2  of \cite{BFGMPV} by a different method.  

As a direct consequence of Lemma \ref{lem:uvt}, we can describe minimizers of the $1$-frame potential as follows. 

\begin{corollary}\label{cor:$1$-frame}
For   any $\mu\in\mathcal{P}(\mathbb{S}^{d-1})$,

$$\int_{\mathbb{S}^{d-1}}\int_{\mathbb{S}^{d-1}}\int_{\mathbb{S}^{d-1}} |\langle x,y \rangle \langle x,z \rangle \langle y,z \rangle|\, d\mu(x) d\mu(y) d\mu(z) \geq \frac 1 {d^2}$$
and the equality holds if and only if $\mu$ is isotropic and $\langle x,y \rangle \langle x,z \rangle \langle y,z \rangle \geq 0$ for any points $x,y,z$ in the support of $\mu$.
\end{corollary}

Using this corollary, we can solve the problem for $0<p<1$ and show that the orthonormal basis is a unique (up to central symmetry and rotations) minimizer.

\begin{theorem}\label{thm:p<1}
For   $0<p<1$,
$$\int_{\mathbb{S}^{d-1}}\int_{\mathbb{S}^{d-1}}\int_{\mathbb{S}^{d-1}} |\langle x,y \rangle \langle x,z \rangle \langle y,z \rangle|^p\, d\mu(x) d\mu(y) d\mu(z) \geq \frac 1 {d^2}$$
and the equality holds if and only if $\mu$ is a uniform distribution over an orthonormal basis (up to central symmetry).
\end{theorem}

\begin{proof}
For $0<p<1$ and any $\mu\in\mathcal{P}(\mathbb{S}^{d-1})$,
$$|\langle x,y \rangle \langle x,z \rangle \langle y,z \rangle|^p\geq |\langle x,y \rangle \langle x,z \rangle \langle y,z \rangle|,$$
so the lower bound follows from Corollary \ref{cor:$1$-frame}. The bound is sharp if and only if $\mu$ minimizes the three-point $1$-frame energy and for all $x,y,z$ from $\supp(\mu)$, $ \langle x,y \rangle \langle x,z \rangle \langle y,z \rangle  =0$ or $1$. This condition is not satisfied for the triple $(x,x,y)$ if there are two points $x$, $y$ such that $\langle x,y \rangle \not\in \{ 0, -1, 1\}$. Therefore, any distinct points in $\supp(\mu)$ are orthogonal or antipodal. The support of an isotropic   $\mu$ must  span $\mathbb R^d$ so, due to (\ref{eqn:isotropic}), $\mu( \{ e_j, -e_j\}) = \frac{1}{d}$ for $j = 1, ..., d$, for some orthonormal basis $e_1, ..., e_d$.
\end{proof}

The main difficulty in solving the problem for $p=1$ is in characterizing all isotropic measures such that $\langle x,y \rangle \langle x,z \rangle \langle y,z \rangle \geq 0$ for any points $x,y,z$ in their support. Note that Rosenfeld sets are discrete miminizers of the $1$-frame energy as they are tight frames (Corollary \ref{cor:rosenframe}) and satisfy this condition. The general problem of describing all minimizers is highly nontrivial since, as mentioned above, even the problem of characterizing all Rosenfeld sets is wide-open for $d\geq 5$. However, we can provide a complete description of minimizers for $d=2$.

\begin{theorem}\label{thm:2-dim}
For any $\mu\in\mathcal{P}(\mathbb{S}^1)$,
$$\int_{\mathbb{S}^1}\int_{\mathbb{S}^1}\int_{\mathbb{S}^1} |\langle x,y \rangle \langle x,z \rangle \langle y,z \rangle|\, d\mu(x) d\mu(y) d\mu(z) \geq \frac 1 {4}$$
and the equality holds if and only if $\mu$ is a convex combination of uniform distributions over two (not necessarily distinct) orthonormal bases (up to central symmetry). 
\end{theorem}

\begin{proof}
Let $\mu$ be a minimizer and assume there is $x\in\supp(\mu)$ such that  $x^{\perp} \cap \supp(\mu) = \emptyset$. Note that the potential is invariant under central symmetry for any of its arguments. This means that we can rearrange the measure by switching points to their opposites so that $\supp(\mu)$ is entirely contained in an open half-circle centered at $x$. Assume the left endpoint of the support now is $x_1$ and the right endpoint is $x_2$. Since both $\langle x, x_1\rangle$ and $\langle x, x_2\rangle$ are positive, $\langle x_1,x_2\rangle \geq 0$, due to Corollary \ref{cor:$1$-frame}. Take the midpoint $y$ of a circular arc between $x_1$ and $x_2$. The angle between $x_1$ and $x_2$ is no greater than $\frac {\pi} 2$ so for any $z\in\supp(\mu)$, $\langle z, y\rangle \geq \frac 1 {\sqrt{2}}$. By (\ref{eqn:isotropic}),
$$\int_{\mathbb{S}^1} |\langle z, y\rangle|^2\,  d\mu(z)=\frac 1 2.$$
This may happen only if $\langle z, y\rangle$ is precisely $\frac 1 {\sqrt{2}}$ for any $z\in\supp(\mu)$ and $\langle x_1, x_2 \rangle =0$. However, $x$ cannot coincide with $x_1$ or $x_2$ because it does not have an orthogonal counterpart in the support and, therefore, $\langle x, y\rangle$ is definitely not $\frac 1 {\sqrt{2}}$ which contradicts our assumption.

We conclude that any $x\in\supp(\mu)$ has an orthogonal counterpart. Now we assume that there are at least three pairwise non-orthogonal and pairwise non-opposite points $x, y, z$ in the support of $\mu$. Each of them has an orthogonal couterpart $x^{\perp}$, $y^{\perp}$, $z^{\perp}\in\supp(\mu)$. We claim that among these six points, there are three with the negative product of pairwise scalar products. Without loss of generality, we can assume, by switching to opposites if needed, that pairwise angles between three of the points, say,  $x$, $y$, $z$ are all acute and $y$ is between $x$ and $z$ on the circle. Then $\langle x, y^{\perp}\rangle \langle x, z\rangle \langle y^{\perp}, z\rangle < 0$, a contradiction, so we may have no more than two pairwise non-orthogonal and pairwise non-opposite points.

Again switching to opposite points if necessary, we assume that the support of $\mu$ consists of only four points (the case of two points in handled later): $x$,  $y$, and their orthogonal counterparts $x^{\perp}$ and  $y^{\perp}$. Assume $\mu(x)=\alpha_1$, $\mu(x^{\perp})=\alpha_2$, $\mu(y)=\beta_1$, $\mu(y^{\perp})=\beta_2$. Using (\ref{eqn:isotropic}) for an arbitrary $w \in \mathbb S^{d-1}$, we get
$$\frac 1 2 = \int_{\mathbb{S}^1} |\langle z, w\rangle|^2\,  d\mu(z) = \alpha_1 \langle w, x\rangle^2 + \alpha_2 (1-\langle w, x \rangle^2) + \beta_1 \langle w, y\rangle^2 + \beta_2 (1-\langle w, y\rangle^2).$$
Hence there exists a linear dependence between 1, $\langle w, x\rangle^2$, and $\langle w, y\rangle^2$. This may happen only when $y$ is orthogonal or opposite to $x$ or the dependence is trivial. In the latter case, $\alpha_1=\alpha_2$ and $\beta_1=\beta_2$ so $\mu$ is precisely a convex combination of uniform distributions over two orthonormal bases.

The remaining case is when the support of $\mu$ consists of only $x$ and $x^{\perp}$ and, by the same condition, their weights must be equal.
\end{proof}

There are certain distinctions between the behavior of the $p$-frame potentials in the two- and three-input cases: in particular, while in the two-input case the uniform distribution over an orthonormal basis is the unique (up to symmetries) minimizer of the $p$-frame energy in the range $0<p<2$, the discussion in this section suggests that for the three-input energy this is only true for $0<p<1$, as Theorem \ref{thm:2-dim} indicates that one does not have uniqueness for $p =1$, and an argument similar to Theorem \ref{thm:p<1} would then show that the orthonormal basis is not a minimizer for $p> 1$. This difference might be partially explained by the fact that the degree of each of the variables in the $p$-frame potential in the three-input case is twice as large.

In the two-input case, the general conjecture in \cite{BGMPV2} claims that all minimizers of the  $p$-frame energy are discrete when $p\notin 2\mathbb{N}$.  It is therefore natural  to conjecture that all minimizers of the three-point $p$-frame energy are also discrete, at least  when $p\notin\mathbb{N}$.

\section{Isotropic measures without obtuse angles}\label{sec:non-obtuse}

Isotropic measures with no obtuse angles between any two points in their support clearly minimize the three-point $1$-frame energy (due to Corollary \ref{cor:$1$-frame}). In this section we show that only orthonormal bases satisfy this condition.

\begin{theorem}\label{thm:non-obtuse}
If $\langle x, y\rangle \geq 0$ for any $x$, $y$ in the support of an isotropic measure $\mu\in\mathcal{P}(\mathbb{S}^{d-1})$, then $\mu$ is a uniform distribution over an orthonormal basis.
\end{theorem}

For the proof of this theorem, we need three ingredients: the spherical Jung inequality, the procedure of lifting of balanced isotropic measures to a higher dimension, and the linear programming bound on the diameter of a balanced isotropic measure.

Jung's inequality \cite{Jung} shows that the simplex has the maximal circumradius for  a Euclidean set of a given diameter. For the first ingredient of the proof, we use the spherical version of this inequality proved by Dekster \cite{D}.

\begin{theorem}[Dekster]\label{thm:dekster}
If the spherical distance, i.e. $\arccos(\langle x, y \rangle)$, between any two points of a compact set $C \subset\mathbb{S}^{d-1}$ is not greater than $D$, $0\leq D\leq 2\sin^{-1}\sqrt{\frac d {2d-2}}$, then $C$ can be covered by a spherical cap whose radius is a circumradius of a regular simplex in $\mathbb{S}^{d-1}$ with spherical edge length $D$.
\end{theorem}

Note that a regular simplex with edge length $\frac {\pi} 2$ is formed by the endpoints of an orthonormal basis $e_1,\ldots, e_d$ in $\mathbb{S}^{d-1}$. Denote $z=\frac 1 {\sqrt{d}} (e_1+\ldots +e_d)$. The points on the circumsphere of the simplex are defined by $\langle x,z \rangle =\frac 1 {\sqrt{d}}$. Theorem \ref{thm:dekster} then implies that for any spherical set, where $\langle x, y\rangle \geq 0$ for any pair of points $x$ and $y$, there exists $z\in\mathbb{S}^{d-1}$ such that $\langle x, z\rangle \geq \frac 1 {\sqrt{d}}$ for any $x$ in the set.

For the next ingredient of the proof of Theorem \ref{thm:non-obtuse}, we describe the procedure of lifting of an isotropic measure to a higher dimension. Assume $z\in\mathbb{S}^d$ and take for $\mathbb{S}^{d-1}$ the intersection of $\mathbb{S}^d$ and the hyperplane through the origin orthogonal to $z$. Define $f:\mathbb{S}^{d-1}\rightarrow\mathbb{S}^d$ as follows:
$$f(x)=\sqrt{\frac d {d+1}} x + \frac 1 {\sqrt{d+1}} z.$$
For $\mu\in\mathcal{P}(\mathbb{S}^{d-1})$ define the lifted measure $\mu_l\in\mathcal{P}(\mathbb{S}^d)$ as the pushforward of $\mu$ under $f$.

\begin{lemma}\label{lem:lift}
A probability measure $\mu$ is a balanced isotropic measure in $\mathbb{S}^{d-1}$ if and only if $\mu_l$ is an isotropic measure in $\mathbb{S}^{d}$.
\end{lemma}

\begin{proof}
Any $y\in\mathbb{S}^d$ can be represented as $\alpha_y y' + \beta_y z$, where $y'\in\mathbb{S}^{d-1}$ and $\alpha_y^2+\beta_y^2=1$. For any $x' \in\supp(\mu)$, letting $x = f(x')$, $\alpha_x=\sqrt{\frac d {d+1}}$ and $\beta_x=\frac 1 {\sqrt{d+1}}$. Then
\begin{align*}
\int_{\mathbb{S}^d} \langle x, y \rangle^2\, d\mu_l(x) & = \int_{\mathbb{S}^d} \left(\sqrt{\frac d {d+1}}\alpha_y \langle x', y' \rangle + \frac 1 {\sqrt{d+1}}\beta_y\right)^2\, d\mu_l(x) \\
&= \frac 1 {d+1} \int_{\mathbb{S}^d} \left(d \alpha_y^2 \langle x', y' \rangle^2 + 2\sqrt{d}\alpha_y\beta_y \langle x', y' \rangle + \beta_y^2\right)\, d\mu_l(x)\\
& = \frac 1 {d+1} \int_{\mathbb{S}^{d-1}} \left(d \alpha_y^2 \langle x', y' \rangle^2 + 2\sqrt{d}\alpha_y\beta_y \langle x', y' \rangle + \beta_y^2\right)\, d\mu(x')\\
& =\frac {d \alpha_y^2} {d+1} \int_{\mathbb{S}^{d-1}} \langle x', y' \rangle^2\, d\mu(x') + \frac {2\sqrt{d}\alpha_y\beta_y} {d+1} \int_{\mathbb{S}^{d-1}} \langle x', y' \rangle\, d\mu(x')+ \frac {\beta_y^2} {d+1}.
\end{align*}

If $\mu$ is balanced and isotropic, then, due to (\ref{eqn:isotropic}), for any $y'\in\mathbb{S}^{d-1}$, 
\begin{align*}
\frac {d \alpha_y^2} {d+1} \int_{\mathbb{S}^{d-1}} \langle x', y' \rangle^2\, d\mu(x') + \frac {2\sqrt{d}\alpha_y\beta_y} {d+1} \int_{\mathbb{S}^{d-1}} \langle x', y' \rangle\, d\mu(x')+ \frac {\beta_y^2} {d+1} &=\frac {d \alpha_y^2} {d+1}\cdot \frac 1 d + 0 + \frac {\beta_y^2} {d+1}\\
& = \frac {\alpha_y^2+\beta_y^2} {d+1} = \frac 1 {d+1},
\end{align*}
so (\ref{eqn:isotropic}) is satisfied and $\mu_l$ is isotropic in $\mathbb{S}^{d}$. 

For the other direction, assume $\mu_l$ is isotropic. Taking $y\in\mathbb{S}^{d-1}$, i.e. $\alpha_y=1$ and $\beta_y=0$, we get
$$\frac 1 {d+1} = \frac d {d+1} \int_{\mathbb{S}^{d-1}} \langle x', y' \rangle^2\, d\mu(x'),$$
so, again by (\ref{eqn:isotropic}), $\mu$ is istropic in $\mathbb{S}^{d-1}$. Now we use that both $\mu_l$ and $\mu$ are isotropic and take an arbitrary $\alpha_y\neq 0, 1$ to get
$$\frac 1 {d+1} = \frac {d \alpha_y^2} {d+1}\cdot \frac 1 d + \frac {2\sqrt{d}\alpha_y\beta_y} {d+1} \int_{\mathbb{S}^{d-1}} \langle x', y' \rangle\, d\mu(x')+ \frac {\beta_y^2} {d+1}$$
so
$$\frac {2\sqrt{d}\alpha_y\beta_y} {d+1} \int_{\mathbb{S}^{d-1}} \langle x', y' \rangle\, d\mu(x')=0,$$
meaning that $\mu$ is balanced.
\end{proof}

\begin{remark}
Essentially the same lifting construction for discrete sets was used, for example, in \cite{BGOY} to construct two-distance tight frames and in \cite{Ball1} to make use of the geometric Brascamp-Lieb inequality.
\end{remark}

The next result provides the bound on the diameter of the support of a balanced isotropic measure confined to a sphere. The spirit of the proof resembles the linear programming approach, where a carefully constructed polynomial typically leads to a required bound. Similar optimization results for isotropic measures were also obtained in \cite{Gl}.

\begin{theorem}\label{thm:iso_diam}
Let $\mu\in\mathcal{P}(\mathbb{S}^{d-1})$ be a balanced isotropic measure. If $\langle x,y \rangle \geq -\frac 1 d$ for any $x,y\in\supp(\mu)$, then $\mu$ is a uniform distribution over a regular simplex.
\end{theorem}

\begin{proof}
Define $P(t)=(t-1)(t+\frac 1 d)$ and note that $P(\langle x, y \rangle) \leq 0$ for any $x,y\in\supp(\mu)$. Then
\begin{equation}\label{eqn:iso_diam}
\int_{\mathbb{S}^{d-1}} \int_{\mathbb{S}^{d-1}} P(\langle x, y \rangle)\, d\mu(x)d\mu(y) \leq 0
\end{equation}
and it is strictly less than 0 if $P(\langle x, y \rangle)<0$ for some $x,y\in\supp(\mu)$.

On the other hand, $P(\langle x, y \rangle) = \langle x, y \rangle^2 - \frac {d-1} d \langle x, y \rangle - \frac 1 d$. Since $\mu$ is balanced and isotropic,
$$\int_{\mathbb{S}^{d-1}} \int_{\mathbb{S}^{d-1}} \langle x, y \rangle\, d\mu(x)d\mu(y) = 0 \, \text{ and }\, \int_{\mathbb{S}^{d-1}} \int_{\mathbb{S}^{d-1}} \langle x, y \rangle^2\, d\mu(x)d\mu(y) = \frac 1 d.$$
This means (\ref{eqn:iso_diam}) is sharp for $\mu$ and $P(\langle x, y \rangle)=0$ for any $x,y\in\supp(\mu)$. In particular, $\langle x, y \rangle$ must be $-\frac 1 d$ for any distinct $x$ and $y$ in the support of $\mu$. Therefore, $\supp(\mu)$ is a subset of the $d+1$ vertices of a regular simplex in $\mathbb{S}^{d-1}$. Among linear combinations of these vertices, only those with equal coefficients are 0. Given that $\mu$ is balanced, it must be a uniform distribution over the set of vertices of a regular simplex.
\end{proof}

\begin{remark}\label{rem:iso_diam}
Theorem \ref{thm:iso_diam} implies that the spherical diameter of the support of a balanced isotropic measure on the unit sphere is at least $\arccos \left( -\frac 1 d\right)$. Moreover, if the diameter is precisely $\arccos \left(-\frac 1 d\right)$, the measure is necessarily a uniform distribution over a regular simplex.
\end{remark}

We now have all the ingredients necessary for the proof of Theorem \ref{thm:non-obtuse}.

\begin{proof}[Proof of Theorem \ref{thm:non-obtuse}]
For the first step, we use Theorem \ref{thm:dekster} and, as described above, find $z\in\mathbb{S}^{d-1}$ such that $\langle x, z \rangle\geq\frac 1 {\sqrt{d}}$ for all $x\in\supp(\mu)$. By Equality (\ref{eqn:isotropic}),
$$ \int_{\mathbb{S}^{d-1}} \langle x, z \rangle^2\, d\mu(x) = \frac 1 d,$$
so $\langle x, z \rangle = \frac 1 {\sqrt{d}}$ for all $x\in\supp(\mu)$.

This means $\mu$ is a lifted measure for $\mu'\in\mathcal{P}(\mathbb{S}^{d-2})$. By Lemma \ref{lem:lift}, $\mu'$ is balanced and isotropic. If $x, y\in\supp(\mu)$, then $x=\sqrt{\frac {d-1} d} x' + \frac 1 {\sqrt{d}} z$ and $y=\sqrt{\frac {d-1} d} y' + \frac 1 {\sqrt{d}} z$, where $x', y'\in\supp(\mu')$. The condition $\langle x,y \rangle \geq 0$ is equivalent to $\langle x',y' \rangle \geq -\frac 1 {d-1}$. Then $\mu'$ satisfies Theorem \ref{thm:iso_diam} and must be a uniform distribution over a regular simplex in $\mathbb{S}^{d-2}$. Lifting the simplex to $\mathbb{S}^{d-1}$ we get that $\mu$ must be a uniform distribution over an orthonormal basis.
\end{proof}

\begin{remark}\label{rem:non-obtuse}
Similarly to Remark \ref{rem:iso_diam}, Theorem \ref{thm:non-obtuse} implies that the spherical diameter of the support of an isotropic measure on the unit sphere is at least $\frac \pi 2$ and, if the diameter is precisely $\frac \pi 2$, the measure is necessarily a uniform distribution over an orthonormal basis.
\end{remark}

\section{\texorpdfstring{Support of a minimizer for $1$-frame energy has dimension not greater than $d-2$}{Support of a minimizer for 1-frame energy has dimension not greater than d-2}}\label{sec:supp}

One of the main results of \cite{BGMPV2} states that the support of a minimizer of a two-point $p$-frame energy has empty interior, when $p$ is not an even integer. Here we prove a similar result for the three-point $1$-frame energy.

\begin{lemma}\label{lem:orth_point}
If $\mu$ is a minimizer of the three-point $1$-frame energy among all measures in $\mathcal{P}(\mathbb{S}^{d-1})$, then for any $x\in\supp(\mu)$ there is $y\in\supp(\mu)$ such that $\langle x,y \rangle = 0$.
\end{lemma}

\begin{proof}
Assume there is $x$ in the support of $\mu$ such that none of the other points in $\supp(\mu)$ are orthogonal to $x$. Rearrange the measure so that all points in its support form an acute angle with $x$ by switching points to their opposites if necessary. Now if $y, z\in\supp(\mu)$, both $\langle x,y \rangle$ and $\langle x,z \rangle$ are positive. As was established in Corollary \ref{cor:$1$-frame}, $\langle x,y \rangle \langle x,z \rangle \langle y,z \rangle \geq 0$ for any triple of points in the support of $\mu$, hence, also  $\langle y,z \rangle \geq 0$.  The rearranged measure satisfies Theorem \ref{thm:non-obtuse} so it must be a uniform distribution over an orthonormal basis. Subsequently, the initial assumption is false.
\end{proof}

\begin{theorem}\label{thm:codim1}
If $\mu$ is a minimizer of the three-point $1$-frame energy among all measures in $\mathcal{P}(\mathbb{S}^{d-1})$, then the Hausdorff dimension of $\supp(\mu)$ is no greater than $d-2$.
\end{theorem}

\begin{proof}

Fix $0<\theta<\pi/4$. We consider an arbitrary spherical cap $C$ of radius $\theta$. It is sufficient to show that the $(d-2)$-dimensional Hausdorff content of $C\cap\supp(\mu)$ is finite. Indeed, since the sphere can be covered by a finite number of caps with radius $\theta$, the total $(d-2)$-dimensional Hausdorff content of $\supp(\mu)$ is finite and the dimension of $\supp(\mu)$ is no greater than $d-2$.

Take an arbitrary $x$ from $C\cap\supp(\mu)$. By Lemma \ref{lem:orth_point}, there is $y\in\supp(\mu)$ such that $\langle x,y \rangle = 0$. Assume there are two points $x_1$, $x_2$ in $C\cap\supp(\mu)$ such that $\langle x_1, y\rangle > 0$ and $\langle x_2, y\rangle < 0$. The spherical distance between $x_1$ and $x_2$ is no greater than $2\theta<\pi/2$ so $\langle x_1, x_2\rangle > 0$. Therefore, we have found three points $x_1, x_2, y$ in $\supp(\mu)$ such that $\langle x_1, y \rangle \langle x_2, y \rangle \langle x_1, x_2 \rangle <0$. This contradicts Corollary \ref{cor:$1$-frame}.

This means all points $z\in C\cap\supp(\mu)$ simultaneously satisfy either $\langle z, y\rangle \geq 0$ or $\langle z, y\rangle \leq 0$. We denote by $H_x$ a halfspace defined either by $\langle z, y\rangle \geq 0$ or $\langle z, y\rangle \leq 0$ such that $C\cap\supp(\mu)\subset H_x$. Then $C\cap\supp(\mu)$ is a subset of the convex set $H$ defined as the intersection of $H_x$ taken for all $x\in C\cap\supp(\mu)$. Moreover, each point $x\in C\cap\supp(\mu)$ is a boundary point of $H$ because it belongs to a hyperplane $\langle x, y\rangle = 0$ defining $H_x$. Since $C \cap H$ is convex, Corollary \ref{cor:Convex Boundary Finite} tells us its boundary has finite $(d-2)$-dimensional Hausdorff content so the content of $C\cap\supp(\mu)$ is finite as well.
\end{proof}

\section{Acknowledgments}

We would like to thank David de Laat and Alexandr Polyanskii for useful discussions. All of the authors express gratitude to ICERM for hospitality and support during the Collaborate{@}ICERM  program  in  2021. D.~Bilyk has been supported by the NSF grant DMS-2054606 and Simons  Collaboration Grant 712810. D.~Ferizovi\'{c} thankfully acknowledges support by the Methusalem grant of the Flemish Government. A.~Glazyrin was supported by the NSF grant DMS-2054536. R.W.~Matzke was supported by the Doctoral Dissertation Fellowship of the University of Minnesota, the Austrian Science Fund FWF project F5503 part of the Special Research Program (SFB) ``Quasi-Monte Carlo Methods: Theory and Applications", and NSF Postdoctoral Fellowship Grant 2202877. O.~Vlasiuk was supported by an AMS-Simons Travel Grant.

\section{Appendix}\label{sec:Appendix}

For $s >0$, we denote the $s$-dimensional Hausdorff measure by $\mathcal{H}_s$. The fact that the Hausdorff dimension of the boundary of a convex $d$-dimensional set cannot be larger than $d-1$ seems to be well-known. However, we were not able to find a concrete reference for this statement in the literature so we suggest a short proof here.

\begin{proposition}
If $K \subset \mathbb{R}^d$ is a non-empty, compact, convex set, then $\mathcal{H}_{d-1}(\partial K)$ is finite.
\end{proposition}

\begin{proof}
If there is a set $d+1$ points in $K$ that is not contained in an affine hyperplane, $K$ contains the simplex with these points as vertices, so $K$ has an interior point. If no such subset exists, then $K$ is contained by a hyperplane $H$ of dimension $d-1$, so, since $K$ is compact, $\mathcal{H}_{d-1}(\partial K) \leq \mathcal{H}_{d-1}(K) < \infty$.

Now, without loss of generality, assume $0 \in K^{\circ}$, let $p: \mathbb{S}^{d-1} \rightarrow \partial K $ be the central projection from the unit sphere onto the boundary of $K$, i.e. $p(x) \in x \mathbb{R}_+ \cap \partial K$. Note that this function is a bijection.

For any $x, y \in \mathbb{S}^{d-1}$, we have that
\begin{align*}
\frac{\|p(x)- p(y) \|}{\|x - y\|} 
& = \frac{ \Big\| \| p(x) \| x - \| p(y) \| x + \|p(y)\| x - \|p(y)\| y \Big\|}{\|x - y\|} \\
& \leq  \frac{ \Big| \| p(x) \| - \|p(y)\| \Big| \; \|x \| + \|p(y)\| \; \| x - y\| }{\|x - y\|} \\
& = \frac{\Big| \| p(x) \| - \|p(y)\| \Big|}{\|x - y\|} + \|p(y)\|.
\end{align*}
The radial function of $K$, $r: \mathbb{S}^{d-1} \rightarrow [0, \infty)$ with $r(x) = \|p(x)\|$, is Lipschitz (see, e.g., \cite[Theorem 1]{T}), and $\|p(y)\|$ is bounded, so we can see that the projection $p$ is indeed Lipschitz (in fact, one can show it is bi-lipschitz), and so $\mathcal{H}_{d-1}(\partial K) < \infty$ (see, e.g., \cite[Theorem 7.5]{Ma}).

\end{proof}

If $C$ is a spherical cap of radius $\theta \in (0, \frac{\pi}{2})$ centered at $x$, and $T$ is the hyperplane tangent to the sphere at $x$, then the central projection from $C$ to $T$ is injective, preserves geodesics (and therefore convexity) and is continuously differentiable, with a continuously differentiable inverse on its image (and therefore is bilipschitz). We then have the following corollary:

\begin{corollary}\label{cor:Convex Boundary Finite}
If $K \subset \mathbb{S}^{d-1}$ is geodesically convex and contained by a spherical cap of radius $\theta \in (0 , \frac{\pi}{2})$, then $\mathcal{H}_{d-2}(\partial K)$ is finite.
\end{corollary}

\bibliographystyle{alpha}

\end{document}